\documentclass{amsart}
\usepackage{graphicx}
\usepackage[]{epsfig}
\usepackage{amssymb}
\usepackage[]{pstricks}
\newpsobject{malla}{psgrid}{subgriddiv=1,griddots=10,gridlabels=6pt}

\newtheorem{theorem}{Theorem}[section]
\newtheorem{proposition}[theorem]{Proposition}

\newtheorem{lemma}[theorem]{Lemma}
\newtheorem{remark}[theorem]{Remark}

\newcommand{\p}{\partial}

\newcommand{\C}{{\mathbb C}}

\title[Non-linear bi-algebraic curves and surfaces in \(\mathcal H(12)\) and \(\mathcal H(18)\)]
{Non-linear bi-algebraic curves and surfaces in moduli spaces of Abelian differentials}

\author{Bertrand Deroin \& Carlos Matheus}

\address{\textbf{Bertrand Deroin}
\newline
CNRS \& Universit\'e de Cergy-Pontoise (UMR CNRS 8088), 95302, Cergy-Pontoise, France.} 
\email{bertrand.deroin@cyu.fr}

\address{\textbf{Carlos Matheus}
\newline
CNRS \& \'Ecole Polytechnique (UMR CNRS 7640), 91128, Palaiseau, France.} 
\email{carlos.matheus@math.cnrs.fr}

\date{\today}

\begin{document}
\maketitle

\setcounter{page}{1}

\begin{quote}
{\normalfont\fontsize{8}{10}\selectfont {\bfseries Abstract.} The strata of the moduli spaces of Abelian differentials are non-homogenous spaces carrying natural bi-algebraic structures. Partly inspired by the case of homogenous spaces carrying bi-algebraic structures (such as torii, Abelian varieties and Shimura varieties), Klingler and Lerer recently showed that any bi-algebraic curve in a stratum of the moduli space of Abelian differentials is linear provided that the so-called condition $(\star)$ is fulfilled. 

In this note, we construct a non-linear bi-algebraic curve, resp. surface, of Abelian differentials of genus $7$, resp. $10$. \par}
\end{quote}

\section{Introduction} 

The study of transcendence properties and unlikely intersections in Diophantine geometry is a fascinating topic possessing a vast literature developing in many directions including those related to the interplay between Hodge theory and the geometry of homogenous spaces such as Abelian varieties and their moduli spaces (that is, Shimura varieties). In this context, an impressive number of heuristic principles and rigorous results\footnote{Such as Ax--Schanuel conjectures, Ax--Lindemann type theorems, Andr\'e--Oort and Zilber--Pink conjectures.} was discovered by many authors, and, as a way to unify these statements and also suggest new ones, the point of view of \emph{bi-algebraic structures} became increasingly popular: see, for instance, the survey \cite{KUY} and the articles \cite{BKT} and \cite{BKU}. 

In their recent work \cite{KL}, Klingler and Lerer proposed\footnote{This is a natural goal because the moduli spaces of Abelian differentials seem to ``behave'' like homogenous spaces: for example, the celebrated breakthroughs by Eskin, Mirzakhani and Mohammadi \cite{EM}, \cite{EMM} show that this is the case from the point of view of Dynamical Systems.} to extend the bi-algebraic point of view to the non-homogenous setting of moduli spaces of Abelian differentials. More concretely, the moduli space of Abelian differentials of genus $g$ is stratified into complex quasi-projective algebraic orbifolds $H(\kappa)$ parametrising non-trivial Abelian differentials whose zeroes have multiplicities prescribed by a list $\kappa=(k_1,\dots,k_{\sigma})$ such that $k_1+\dots+k_{\sigma}=2g-2$. As it was proved by Veech and Masur, the relative periods of the elements of $H(\kappa)$ can be used to define the so-called \emph{period charts} inducing a linear integral structure on the analytification of $H(\kappa)$. In particular, after projectivising the stratum $H(\kappa)$, $\kappa=(k_1,\dots,k_{\sigma})$, we obtain a quasi-projective orbifold $\mathcal{H}(\kappa)$ of dimension $2g-2+\sigma$ whose analytification possesses a linear projective structure. In this setting, a closed, irreducible, algebraic subvariety $W$ of $\mathcal{H}(\kappa)$ is \emph{bi-algebraic} if its analytification $W^{\textrm{an}}$ is algebraic in period charts\footnote{That is, the relative periods of Abelian differentials projectively lying in $W^{\textrm{an}}$ satisfy exactly $\textrm{codim}_{\mathcal{H}(\kappa)}(W)$ independent algebraic relations (over $\mathbb{C}$).}, cf. \cite[Def. 1.1]{KL}, and Klingler and Lerer proved that all bi-algebraic curves in the strata $\mathcal{H}(2)$ and $\mathcal{H}(1,1)$ of Abelian differentials of genus two are (projectively) linear (in period charts), and, in general, any bi-algebraic curve in $\mathcal{H}(\kappa)$ is linear \emph{provided} their condition $(\star)$ is satisfied (cf. \cite[Thm. 2.8 \& 2.10]{KL}). Furthermore, they asked whether the bi-algebraicity of a subvariety of $\mathcal{H}(\kappa)$ is already enough to automatically ensure its linearity without extra conditions (cf. \cite[Conjecture 2.7]{KL}). The main results of this paper say that some bi-algebraic subvarieties can be non-linear: 

\begin{theorem}\label{t: non linear bi algebraic curve} The projectivisation of the family \(\{ (C_u, \omega_{u} ) \} _{u\in \C \setminus\{0,\pm 1\}} \) of Abelian differentials defined by 
\begin{equation}\label{eq: curve family} C_u := \overline{\{ y^6 = x (x-1) (x+1) (x-u)\}} \text{ and } \omega_{u} :=  x^2dx/y^5, \end{equation}  
is a bi-algebraic curve in \(\mathcal H(12) \) which is not linear. 
\end{theorem}

\begin{remark} Note that \(C_u\) is a branched cover of \(\overline{\{z^2=x(x-1)(x+1)(x-u)\}}\). In particular, the family of Abelian differentials $\{(C_u, [dx/y^3])\}_{u\in\mathbb{C}\setminus\{0,\pm1\}}$ corresponds to an arithmetic Teichm\"uller curve (in the sense of \cite[\S 5]{GJ}). 
\end{remark}

\begin{theorem}\label{t: non linear bi algebraic surface} Given a generic algebraic curve \(C\subset \mathbb C^3\),  the projectivisation of the family \(\{ (C_{a,b,c}, \omega_{a,b,c} ) \} _{\substack{(a,b,c)\in C, \\ a,b,c\in \C \setminus\{0, 1\} \\ distinct}}\) of Abelian differentials defined by 
\begin{equation}\label{eq: surface family} C_{a,b,c} := \overline{\{ y^6 = x (x-1) (x-a) (x-b) (x-c)\}} \text{ and } \omega_{a,b,c} :=  dx/y^5, \end{equation}  
is a bi-algebraic curve in \(\mathcal H(18) \) which is not linear. 
\end{theorem}

In particular, observe that if \(S\subset \mathbb C^3\) is a generic algebraic surface, then the projectivisation of the family \(\{ (C_{a,b,c}, \omega_{a,b,c} ) \} _{\substack{(a,b,c)\in S, \\ a,b,c\in \C \setminus\{0, 1\} \\ distinct}}\) is a bi-algebraic surface in \(\mathcal H (18)\) which is not linear. We do not have any example of a non linear bi-algebraic subvariety of dimension at least \(3\) in some stratum of abelian differentials. 

The proofs of these statements occupy the rest of the paper. More precisely, after a brief discussion of the variations of Hodge structures associated to the families of curves $C_u$ and $C_{a,b,c}$ in Section \ref{s.fixed-parts}, we establish Theorem \ref{t: non linear bi algebraic curve}, resp. \ref{t: non linear bi algebraic surface}, in Section \ref{s.Thm1}, resp. \ref{s.Thm2}. 

\section*{Acknowledgements}

This work began at the workshop \emph{Finiteness results for special subvarieties: Hodge theory, o-minimality, dynamics} organised by S. Filip, D. Fisher and B. Klingler held in Como (Italy) from October 16 to 22, 2022. We are indebted to several participants of this workshop (including S. Cantat, S. Filip, B. Klingler, L. Lerer and M. M\"oller) for many interesting discussions. 


\section{Eigencohomology of cyclic covers}

Let \(t=(t_1,\ldots, t_n) \in \mathbb C^n \) be a collection of distinct complex numbers. The plane algebraic curve defined by the equation 
\[ y^d = (x-t_1) \ldots (x-t_n) \]
compactifies as a smooth compact Riemann surface \(C_t\) by adding \( a= \textrm{gcd}(d,n)\) points at infinity.  The function \( X_t: C_t \rightarrow  \mathbb P^1 \) defined by \(X_t(x,y)= x\)  is a ramified covering over the sphere having \(n\) critical points of order \(d\) and \( a \) critical point(s) of order \(d / a\). So by Riemann-Hurwitz,  the genus \(g\) of \(C_t\) is given by the formula:
\[ g = \frac{(n-1) (d-1) - (a-1) }{2}. \]  
In particular, if \( (d,n)= (6,4)\),  \( g = 7\), and if \( (d,n)= (6,5)\), \( g= 10\).

The covering \(X_t\) is an abelian covering whose Galois group is generated by the transformation \( \pi_ t (x,y) = (x, \zeta y) \) where \( \zeta = \exp (2i\pi /d)\). We denote by \( H^* _\zeta (C_t, \mathbb C) \) the eigenspace \( \text{Ker} (\pi_t^* - \zeta) \subset H^* (C_t, \mathbb C)\). Unless specified, a holomorphic/meromorphic eigenform on \(C_t\) is a form \(\eta\) such that \(\pi_t^* \eta = \zeta \eta\). The following result is well-known, see, e.g., \cite[\S3]{McM} for more details. We provide the proof for completeness.

\begin{lemma}\label{l: trivialization}
In the regime \( n < d\), we have 
\[ H^{1,0} _{\zeta} (C_t,\mathbb C) = \{ U(x) \frac{dx}{y^{d-1}} \, |\, U \text{ polynomial of degree } \leq n-2 \}.\]
Moreover, the eigenspace \( H^1_{\zeta }(C_t, \mathbb C)\) is made of cohomology classes of holomorphic eigenforms, i.e. 
\[ H^1_{\zeta} (C_t,\mathbb C)\simeq  H^{1,0} _{\zeta}(C_t,\mathbb C)  \text{ and }  H^{0,1} _\zeta (C_t, \mathbb C) = \overline{H^{1,0}_{\zeta^{d-1}} (C_t, \mathbb C)} =0.\] 
\end{lemma}

\begin{proof} A form in \( H^{1,0} _{\zeta} (C_t)\) can be expressed as \( \eta = U(x) \frac{dx} {y^{d-1} } \) where \( U\) is a meromorphic function on \(\mathbb P^1\), which is holomorphic except possibly at the points \( t_i \)'s and the point at infinity. In the sequel assume that \(U\) does not vanish identically. 

Around the point \( t_i\), we can write \( U(x) \sim \text{cst}\, (x-t_i) ^{k_i}\) for some integer \(k_i\) and a non zero constant. We can also take \(y\) as a holomorphic coordinates and we have \( x-t_i \sim \text{cst}\, y^d\), so 
\[ \eta \sim _{(t_i, 0)}  \text{cst}\, y^{dk_i } dy \]
We deduce for \(\eta\) to be holomorphic, \(k_i\) must be non negative for every \(i=1,\ldots, n\), hence \(U\) is a polynomial of \(x\).

It will be convenient to introduce the integers \(b,c\) such that \( n=ab\) and \(d= ac\). Around a point \(\infty\) of \(C_t\) at infinity, we have a coordinate \(z\) so that \(x\sim \text{cst}\, z^{-c}\) and \( y\sim \text{cst}\, z^{-b} \). We then have (denoting \(\text{deg} (U)\) the degree of \(U\)) 
\[ \eta \sim _{\infty}   \text{cst} \, z^{- c \text{deg} (U) - c - 1 +b(d-1) }  \]
 which shows that \(\eta\) is holomorphic iff \(\text{deg} (U) \leq \frac{b(d-1) - c - 1}{c}=  n - 1- \frac{ (b+1) }{c} \), or equivalently, since \( b<c\), iff \( \text{deg} (U) \leq n-2\). In particular \( \text{dim}  \, H^{1,0} _{\zeta} (C_t, \mathbb C) = n-1 = \text{dim}\,  H^1 _\zeta (C_t,\mathbb C)\) so we are done. \end{proof}


\section{Finite monodromies}\label{s.fixed-parts} 

The condition $(\star)$ of Klingler and Lerer (cf. \cite[Def. 6.5]{KL}) indicates that the families of curves whose Jacobians possess non-trivial fixed parts are potential sources of non-linear bi-algebraic subvarieties: see \S6.4 of \cite{KL} for more explanations. In this direction, we note that such families of curves were previously studied by several authors in the context of Teichm\"uller dynamics (e.g., \cite{FMZ}, \cite{MY}, \cite{Mo}, \cite{McM}, \cite{AMY} among many others), and, in fact, the families of curves $C_u$ and $C_{a,b,c}$ are particular instances of the families of curves described in \cite[Thm. 8.3]{McM}. For later reference, we recall below some of the remarkable properties of the families of Jacobians of $C_u$ and $C_{a,b,c}$. 





\begin{proposition}\label{p: omega2 bi algebraic}
If \((d,n) = (6,4)\) or \( (d,n)= (6,5)\), the monodromy of the Gauss-Manin connection on the bundle \(H^1_\zeta\)  is finite.  In particular, the families of forms \( \{(C_u,[\omega_{u}])\}_{u\neq 0,\pm1} \), lying in \(\mathcal H(12)\), and \( \{(C_{a,b,c},[\omega_{a,b,c}])\}_{\substack{a,b,c\neq 0,1 \\ distinct}} \), lying in \(\mathcal H(18)\), are bi-algebraic.  
\end{proposition} 

Here \([\omega]\) is the projective class of \(\omega\) in the projectivization of the Hodge bundle.

\begin{proof} This statement is included in Theorem 8.3 of \cite{McM}. Let us provide the proof for completeness. Lemma \ref{l: trivialization} ensures that 
the Gauss-Manin invariant hermitian form  \(\frac{i}{2} \int \omega\wedge \overline{\omega} \) on \( H^1 _\zeta \) (resp. \(H^1 _{\zeta^5} =\overline{H^1 _\zeta}\)) is positive definite, so its monodromy is unitary. Moreover,  
the bundle \(H^1(C_t,\mathbb{C})_{\zeta}\oplus H^1(C_t,\mathbb{C})_{\zeta^5}\) is defined over $\mathbb{Q}$ (because $\zeta$, $\zeta^5$ are the sole primitive sixth roots of unity). So the monodromy is integral (it is the restriction of the symplectic integral representation to a rational subspace) and preserves a unitary form; hence it is finite. In particular, as it is explained in \cite[pp. 929]{McM} for instance, since the periods of the corresponding families of forms are holomorphic (multivalued) functions of $t$ which are algebraic when the monodromy is finite, one has that the corresponding families of forms are bi-algebraic.
\end{proof}





\section{Picard-Fuchs theory: computation of the Gauss-Manin connection}

The vector bundle \( H^1 _\zeta\) over the configuration space \(B\) of $n$ distinct complex numbers is isomorphic, thanks to Lemma  \ref{l: trivialization}, to the product \(B\times \mathbb C_{n-2} [x]\), where $\mathbb C_{n-2} [x]$ stands for the space of polynomials of degree $\leq n-2$.  More precisely,  for \( U\in \mathbb C_{n-2} [x]\), we denote  
\begin{equation} \label{eq: holomorphic eigneform} \omega_{t,U}= U(x) \frac{dx}{y^{d-1}} \end{equation} 
 the corresponding holomorphic eigenform on \(C_t\). The map \( (t , U)  \mapsto \omega_{t,U} \) gives the desired trivialization. In the following, we compute the Gauss-Manin connection on \(H^1_\zeta\) in this latter: it is defined as follows 
 \[ \nabla_{\cdot}  ( \omega_{t, U}) = \omega_{t, \nabla_{\cdot} U} .\] 
 We denote 
\begin{equation} \label{eq: polynomial} P(x) = (x-t_1) \ldots (x-t_n).\end{equation}

\begin{proposition} \label{p: Gauss-Manin} 
In the trivialization \( H^1 _\zeta \simeq B \times \mathbb C_{n-2} [x]\) given by Lemma \ref{l: trivialization}, the Gauss-Manin connection is given by  
\[ \nabla _{\partial_{t_k}} U = - \frac{U(t_k)}{P'(t_k)} \frac{P(x)- P(t_k) - P'(t_k) (x-t_k) }{(x-t_k)^2 } +\frac{1}{d} \frac{ U(t_k)}{P'(t_k)} \frac{P'(x) - P'(t_k) } {x-t_k} +\]
\[+\frac{d-1}{d}  \frac{U(x) - U(t_k) }{x-t_k} \]
for any \(U\in \mathbb C_{n-2} [x]\). 
\end{proposition}

\begin{proof} Let \(N_i\) be simply connected neighborhoods of the \(t_i\)'s, chosen sufficiently small to be disjoint. Denoting by \(N= \cup N_i\), and \(Y = \mathbb C \setminus N\), the Riemann surfaces \( X_{t'} ^{-1} (Y)  \) for  \(t'= (t_1',\ldots, t_n ') \in \prod _ i N_i \), are naturally identified to \( X_{t} ^{-1} (Y)  \) via a holomorphic family of biholomorphisms \( \Phi _{t' ,t} : X_{t} ^{-1} (Y)\rightarrow X_{t'} ^{-1} (Y)\) that satisfy \( X_{t'} \circ \Phi _{t' ,t} = X_t\) and \(\Phi_{t,t}= \text{id} \). Using this family it is possible to derivate a holomorphic family of holomorphic \(1\)-forms \(\omega_t\) on \(X_t^{-1}(Y)\) with respect to the \(t_k\)-variables, by the formula 
\[ \partial _{t_k} \omega_t := \left( \partial_ {t'_k} \Phi_{t',t} ^* \omega_{t'} \right)_{|t'=t}  ; \]
by shrinking $N$, this defines a meromorphic \(1\)-form on \(C_t\) which is holomorphic apart from  the points \(  (t_i,0)\in C_t\) but that might develop a pole at  each \( (t_i,0)\in C_t\). These poles have no residue at the points \((t_i,0)\), so the period of \(\partial_{t_k}  \omega_t\) is well-defined as an element  of \( H^1 (C_t, \C)\): for more discussion about this construction (leading to the Kodaira-Spencer map and the Gauss-Manin connection), see Clemens' book \cite[\S 2.10]{Cl} and Voisin's book \cite[Ch. 9]{V}. 

Applied to the family of eigenforms \( \{ (C_t,\omega_{t,U})\}_{t\in \mathbb C^n \setminus \Delta} \), where \( U \in \mathbb C_{n-2} [x]\),  the derivative \( \partial _{t_k} \omega_t \) is a meromorphic eigenform, which is cohomologous to a holomorphic eigenform \(\omega_{t, V_t} \) with \(V_t \in \mathbb C_{n-2} [x]\) by Lemma \ref{l: trivialization}. By definition 
\begin{equation} \label{eq: nabla in trivialization} \nabla _{t_k} U := V_t .\end{equation}
This equation is characterized by the existence of a (unique) meromorphic eigenfunction \( f : C_t\rightarrow \mathbb P^1\) which is such that 
\[ \partial _{t_k} \omega_{t,U} - \omega _{t, V_t} = df . \] 
Here eigenfunction means \( f\circ \pi_t = \zeta f\). Let us now make the computations and find the function \(f\).

Derivating the multi-valued function \( y = \prod_j (x-t_j) ^{1/d} \) with respect to the variables \(t_k\) yields \(\partial_{t_k} y= \frac{-y}{d(x-t_k)} \)
from which we get 
\begin{equation}\label{eq: derivative}  \partial _{t_k} \left( U(x)  \frac{dx}{y^{d-1} }\right) =\frac{d-1}{d} \frac{U(x) }{x-t_k} \frac{dx}{y^{d-1} } .\end{equation}
This meromorphic eigenform has a pole at the point \( (t_k, 0)\), which in the coordinates \(y\) of \(C_t\) at \((t_k,0)\) has the following local behaviour (assuming that \(U(t_k)\neq 0\))  
\begin{equation}\label{eq: local behaviour form}  \partial _{t_k} \left( U(x)  \frac{dx}{y^{d-1} }\right) \sim _{(t_k,0)} (d-1) \frac{U(t_k)} {y^d} dy . \end{equation}

The meromorphic function \( f : C_t \rightarrow \mathbb P^1\)  defined by 
\begin{equation} \label{eq: meeromophic function} f(x,y) := \frac{y}{x-t_k}\end{equation}
has a unique pole at the point \((t_k, 0)\) and it has the following local behaviour 
\begin{equation} \label{eq: local behaviour function} f \sim _{(t_k,0)} \frac{P'(t_k) }{y^{d-1}} \end{equation}
(it has no pole at infinity since \(d\geq n\)). In particular, the eigenform 
\[ \partial _{t_k} \left( U(x)  \frac{dx}{y^{d-1} }\right) + \frac{U(t_k)}{P'(t_k)}df \] 
is holomorphic on \(C_t\), we thus have 
\begin{equation}\label{eq: correction}  \partial _{t_k} \left( U(x)  \frac{dx}{y^{d-1} }\right) + \frac{U(t_k)}{P'(t_k)}df = \omega_{t,V_t} ,\end{equation}  
with \( \nabla_{t_k} U= V_t\) by \eqref{eq: nabla in trivialization}, so it remains to compute \eqref{eq: correction}. We have 
\begin{small} 
\[  df = f \frac{df}{f} =  \frac{y}{x-t_k} \left( \frac{dy}{y} - \frac{dx}{x-t_k} \right) =\frac{y}{x-t_k}\left( \frac{1}{d}\sum _{j\neq k}\frac{dx}{x-t_j} +\left(\frac{1}{d}- 1\right) \frac{dx}{x-t_k} \right)=\]
\[= \frac{P(x)}{x-t_k} \left( \frac{1}{d}\sum _{j\neq k}\frac{1}{x-t_j} - \frac{d-1}{d} \frac{1}{x-t_k} \right) \frac{dx}{y^{d-1}}.\]
\end{small} 
So using the identity 
\[ \frac{P(x)}{x-t_k}\sum_{j\neq k} \frac{1}{x-t_j} = \left(\frac{P(x)}{x-t_k}\right)' = \frac{P'(x)}{x-t_k}-\frac{P(x)}{(x-t_k)^2}\]
we deduce 
\[ V_t (x)= \frac{d-1}{d}\left( \frac{U(x)}{x-t_k} - \frac{U(t_k)}{P'(t_k) } \frac{P(x)}{(x-t_k)^2}\right) +\frac{1}{d} \frac{U(t_k) }{P'(t_k) } \left(\frac{P'(x)}{x-t_k}-\frac{P(x)}{(x-t_k)^2}\right). \]
Rearranging terms and using \(P(t_k)=0\) we find the expression of the Gauss-Manin connection given by Proposition \ref{p: Gauss-Manin}.
 \end{proof}

\section{Proof of Theorem \ref{t: non linear bi algebraic curve}}\label{s.Thm1}

In view of Proposition \ref{p: omega2 bi algebraic}, our task is reduced to check that the lines spanned by $\omega_{u}$, $u\neq 0, 1$, sit ``generically'' on $H^1 _\zeta (C_u, \mathbb C)$. A direct application of Proposition \ref{p: Gauss-Manin} shows that  the Gauss-Manin on \(H^1 _\zeta\) in the trivialization given by the sections \( p_k= x^k \frac{dx}{y^5}\) for \(k=0,1,2\) is  
\begin{equation}\label{eq: Gauss Manin on V} \nabla = d + \frac{dt}{6t(1-u^2)} 
\begin{pmatrix} 
5u^2 -4   & u        & u^2  \\  
5u          & 5u^2  & 5u   \\
 2           & 2u      & 2u^2  
\end{pmatrix}
\end{equation}
By Proposition \ref{p: omega2 bi algebraic}, it suffices to prove that the family of periods \( \{p_{2,u}\}_{u\neq 0, \pm 1}\) does not lie in a flat strict subbundle of \( H^1 _\zeta\), which will be achieved by proving that at a generic parameter \(u\) the periods \(p_{2,u}, \nabla_ u p_{2,u}, \nabla_u ^2 p_{2,u}\) are linearly independent in \(H^1_\zeta(C_u,\mathbb C)\). We use the basis \( p_{k,u}\) of \(H^1_\zeta(C_u, \mathbb C)\) to perform the computations, together with the explicit expression of the Gauss-Manin connection \eqref{eq: Gauss Manin on V}. We have in this basis 
\[p_{2,u} = \begin{pmatrix} 
0\\
0\\
1
\end{pmatrix}, \ \  
\nabla_u p_{2,u} =\frac{1}{6u(1-u^2)} 
\begin{pmatrix} 
u^2  \\  
5u   \\
2u^2  
\end{pmatrix} 
\]
and 
\[
\nabla _u\left( 6(1-u^2) \nabla_u p_{2,u} \right)= \frac{1}{6(1-u^2)}
\begin{pmatrix} 
7+u^2  \\  
40u  \\
24-8u^2  
\end{pmatrix}
.\]
Those three vectors are generically linearly independent, and by Leibniz rule, so are the vectors \(p_{2,u}, \nabla_ u p_{2,u}, \nabla_u^2 p_{2,u}\). 

\begin{remark} The family $(\{y^{10}=x(x-1)(x-v)\}, [dx/y^9])_{v\neq 0,1}$ is a bi-algebraic curve in $\mathcal{H}(16)$ because $dx/y^9$ lies in a fixed part of dimension four (cf. \cite[Thm. 8.3]{McM}). Nonetheless, a direct calculation using Proposition \ref{p: Gauss-Manin} reveals that this bi-algebraic curve is linear. 
\end{remark}



\section{Proof of Theorem \ref{t: non linear bi algebraic surface}}\label{s.Thm2}  

In view of Proposition \ref{p: omega2 bi algebraic}, our task is reduced to check that the lines spanned by $\omega_{a,b,c}$, $a,b,c\neq 0, 1$, sit ``generically'' on the fibers of the vector bundle \( H^1 _\zeta\). As it turns out, this fact is a particular case of Theorem 6.1 in \cite{McM}, where a result of Veech and Masur (about the period charts) is cleverly explored to derive that the (period) map taking the configuration of points $\{0,1,a,b,c\}$ to $[dx/y^5]$ is a holomorphic local diffeomorphism.\footnote{Furthermore, McMullen also gets in \cite[Thm. 9.1]{McM} that the periods of $dx/y^5$ are algebraic functions of $(a,b,c)$.} 

In the sequel, we give an alternative proof of the non-linearity of the bi-algebraic surface \(\{ (C_{a,b,c}, \omega_{a,b,c} ) \} _{\substack{a,b,c\in \C \setminus\{0, 1\} \\ distinct}}\) via Proposition \ref{p: Gauss-Manin}: in particular, we do not rely on the results of Veech and Masur, but rather perform direct calculations (which might be of independent interest). 

Proposition \ref{p: Gauss-Manin} shows that  
\[ \nabla_{\partial_u } \frac{dx}{y^5} = \frac{1}{P'(u)} \left( \frac{1}{6} \frac{P'(x)- P'(u) } {x-u}  - \frac{P(x) - P'(u) (x-u) }{(x-u)^2} \right) \frac{dx}{y^5}\] 
(where $u\in\{a,b,c\}$ and $P(x)=x(x-1)(x-a)(x-b)(x-c)$).

At this point, it remains to check that 
\[ \frac{dx}{y^5}, \ \ \nabla_{\partial_a } \frac{dx}{y^5},\ \ \nabla_{\partial_b } \frac{dx}{y^5} , \ \ \nabla_{\partial_c } \frac{dx}{y^5} \] 
are linearly independent. For this sake, let us write them in terms of the basis $\{x^kdx/y^5:k=0, 1, 2, 3\}$. If we denote $\{u,v,w\}=\{a,b,c\}$, then  
\begin{eqnarray*} 
\frac{P'(x)- P'(u) } {x-u} &=& (u^3 -u^2 + u v - u^2 v + u w - u^2 w - 2 v w + 
 u v w) + (u^2 -u + 3 v - u v + 3 w - u w + 3 v w) x \\ &+& (u - 4  - 
    4 v - 4 w) x^2 + 5 x^3,  
\end{eqnarray*} 
\begin{eqnarray*} 
\frac{P(x) - P'(u) (x-u) }{(x-u)^2}  &=& (u^3 -u^2 + u v - u^2 v + u w - u^2 w - v w + 
 u v w) \\ &+& (u^2 -u + v - u v + w - u w + v w) x + (u -1 - v - 
    w) x^2 + x^3, 
\end{eqnarray*} 
so that 
\begin{eqnarray*} 
\frac{1}{6} \frac{P'(x)- P'(u) } {x-u} - \frac{P(x) - P'(u) (x-u) }{(x-u)^2} &=& \frac{1}{6} (5 u^2 - 5 u^3 - 5 u v + 5 u^2 v - 5 u w + 5 u^2 w + 4 v w - 5 u v w) \\ &+& 
 \frac{1}{6} (5 u - 5 u^2 - 3 v + 5 u v - 3 w + 5 u w - 3 v w) x \\ &+&  
 \frac{1}{6} (2 - 5 u + 2 v + 2 w) x^2 - \frac{x^3}{6}.  
\end{eqnarray*} 
In particular, the matrix $M$ whose column vectors are \( \frac{dx}{y^5}, \ \ P'(a) \nabla_{\partial_a } \frac{dx}{y^5},\ \ P'(b) \nabla_{\partial_b } \frac{dx}{y^5} , \ \ P'(c) \nabla_{\partial_c } \frac{dx}{y^5} \) written in the basis $dx/y^5, xdx/y^5, x^2dx/y^5, x^3dx/y^5$ is  
$$M=\tiny{\left(\begin{array}{cccc}1 & \ast &  
 \ast & 
 \ast \\ 0 & 
 \frac{1}{6} (5 a - 5 a^2 - 3 b + 5 a b - 3 c + 5 a c - 3 b c) &  
 \frac{1}{6} (5 b - 5 b^2 - 3 a + 5 a b - 3 c - 3 a c + 5 b c) & 
 \frac{1}{6} (5 c - 5 c^2 - 3 a + 5 a c - 3 b - 3 a b + 5 b c) \\ 0 &  
 \frac{1}{6} (2 - 5 a + 2 b + 2 c) & \frac{1}{6} (2 + 2 a - 5 b + 2 c) & 
 \frac{1}{6} (2 + 2 a + 2 b - 5 c) \\ 0 & -\frac{1}{6} & -\frac{1}{6} & -\frac{1}{6} \end{array}\right)}$$ 
 Since the determinant of $M$ is 
 $$\det(M) = -\frac{91}{216} (a - b) (a - c) (b - c),$$ 
 the proof of Theorem \ref{t: non linear bi algebraic surface} is now complete.  
 
 \begin{remark} The determinant of \( \frac{dx}{y^5}, \ \  \nabla_{\partial_a } \frac{dx}{y^5},\ \ \nabla_{\partial_b } \frac{dx}{y^5} , \ \ \nabla_{\partial_c } \frac{dx}{y^5} \) in the basis $dx/y^5, xdx/y^5, x^2dx/y^5, x^3dx/y^5$, which is somehow more natural, is given by the expression 
 \[ - \frac{91}{216} \frac{(a-b)(a-c)(b-c)}{P'(a)P'(b)P'(c)}= \frac{91}{216}\frac{1}{abc(a-1)(b-1)(c-1)(a-b)(a-c)(b-c)} \]
\end{remark}


\smallskip

\end{document}